\documentclass[12pt,reqno]{amsart}

\usepackage{amsfonts,amsmath,amsthm}
\usepackage{amssymb,epsfig}
\usepackage{cases}
\usepackage[usenames]{color}
\usepackage{enumerate} 

\usepackage{color} 
\definecolor{vert}{rgb}{0,0.6,0}

\usepackage[text={425pt,650pt},centering]{geometry}
\usepackage{graphicx}
\usepackage{tikz}
\usepackage{caption}
\usepackage{epstopdf}
\numberwithin{figure}{section}

\theoremstyle{plain}
\newtheorem{thm}{Theorem}[section]

\newtheorem{quest}{Question}

\newtheorem{lem}[thm]{Lemma}

\newtheorem{prop}[thm]{Proposition}
\theoremstyle{remark}
\newtheorem{rem}{\bf{Remark}}
\numberwithin{equation}{section}

\newcommand{\N}{\mathbb{N}}

\newcommand{\R}{\mathbb{R}}

\newcommand{\T}{\mathbb{T}}
\newcommand{\Z}{\mathbb{Z}}

\newcommand{\cM}{\mathcal{M}}
\newcommand{\cP}{\mathcal{P}}

\newcommand{\cE}{\mathcal{E}}

\newcommand{\Li}{L^{\infty}}

\newcommand{\al}{\alpha}
\newcommand{\gam}{\gamma}
\newcommand{\del}{\delta}
\newcommand{\ep}{\varepsilon}

\newcommand{\lam}{\lambda}

\newcommand{\Del}{\Delta}

\newcommand{\ol}{\overline}

\newcommand{\supp}{{\rm supp}\,}

\newcommand{\Div}{{\rm div}\,}

\newcommand{\osc}{{\rm osc}\,}

\usepackage{varioref}

\begin{document}

\title[Discounted Hamilton--Jacobi equations]
{The Selection problem for\\ 
discounted Hamilton-Jacobi equations: \\
some non-convex cases}

\thanks{
DG'€™s work was partially supported by
start-up and baseline funding from King Abdullah University of Science and Technology;  
HM's work was partially supported by KAKENHI 15K17574, 26287024, 23244015, 16H03948; 
HT'€™s work was partially supported in part by NSF grant DMS-1615944.
}

\author{Diogo A. Gomes}
\address[D. A. Gomes]{
King Abdullah University of Science and Technology (KAUST), CEMSE Division, Thuwal 23955-6900, Saudi Arabia.}
\email{diogo.gomes@kaust.edu.sa}

\author{Hiroyoshi Mitake}
\address[H. Mitake]{
Institute of Engineering, Division of Electrical, Systems and Mathematical Engineering, 
Hiroshima University 1-4-1 Kagamiyama, Higashi-Hiroshima-shi 739-8527, Japan}
\email{hiroyoshi-mitake@hiroshima-u.ac.jp}

\author{Hung V. Tran}
\address[H. V. Tran]
{
Department of Mathematics, 
University of Wisconsin-Madison, Van Vleck hall, 480 Lincoln drive, Madison, WI 53706, USA}
\email{hung@math.wisc.edu}

\date{\today}
\keywords{Nonconvex Hamilton-Jacobi equations; Discounted approximation; Ergodic problems; Nonlinear adjoint methods}
\subjclass[2010]{
35B40, 
37J50, 
49L25 
}

\maketitle

\begin{abstract}
Here, we study the selection problem for the vanishing discount approximation of non-convex, first-order Hamilton-Jacobi equations. While the selection problem is well understood for convex Hamiltonians, the selection problem for non-convex Hamiltonians has thus far not been studied.
We begin our study by examining a 
generalized discounted Hamilton-Jacobi equation.  Next, using an exponential transformation, we apply our methods to 
strictly quasi-convex and to some non-convex Hamilton-Jacobi equations. 
Finally, we examine a non-convex Hamiltonian with flat parts to which our results do not directly apply. In this case, we establish the convergence by a direct approach. 
\end{abstract}




\section{Introduction}

Let  $\T^n = \R^n / \Z^n$ be the standard $n$-dimensional torus and fix a continuous Hamiltonian, $H:\T^n\times \R^d\to \R$. 
Here, we require $H$\ to be \textit{coercive}; that is, 
\[\ 
\lim_{|p| \to \infty} H(x,p)= \infty, \quad \text{uniformly for} \ x \in\
 \T^n. 
\]
 We do not, however, assume convexity. The \textit{ergodic} Hamilton-Jacobi equation is the partial differential equation (PDE)  
\begin{equation}\label{e}
 H(x,Du)=\overline{H} \qquad\text{in} \ \T^n,  \tag{E}   
\end{equation}
and, for $\ep>0$, the corresponding \textit{discounted problem} is
\begin{equation}
\label{de}
\ep u^\ep +H(x,Du^\ep)=0 \qquad \text{in} \ \T^n. 
\tag{$\rm{D}_\ep$}
\end{equation}
In \eqref{e}, the unknown is a pair, $(u,\overline{H} )\in C(\T^n)\times\R$, whereas in \eqref{de}, the unknown is a function, $u^\ep \in C(\T^n)$. In both \eqref{e} and \eqref{de}, we consider solutions in the viscosity sense. Here, we are interested in the {\em vanishing discount limit},  $\ep\to 0$ in \eqref{de}, and in the characterization of the limit, $u$,  of $u^\ep$ as a particular solution of \eqref{e}.  

The problem \eqref{de}  arises in
optimal control theory and differential
game theory where
 $\ep$ 
 is a discount factor.
Moreover,  \eqref{de}   plays an essential role in  the homogenization of first-order Hamilton-Jacobi
equations. For example, in the study of homogenization in \cite{LPV}, the vanishing discount limit is used to construct solutions to  
the ergodic problem. 
The ergodic problem is sometimes called the \textit{cell problem} or the
\textit{additive eigenvalue problem}. 
The PDE \eqref{de} is also called the \textit{discounted approximation}
of the ergodic  problem.
Properties of the solutions of \eqref{e}  
are relevant in dynamical systems, namely in weak Kolmogorov-Arnold-Moser (KAM) theory (see \cite{FaB}), 
and they have applications in the study of the long-time behavior of Hamilton-Jacobi
equations.

In recent years, there was significant progress in the analysis of non-convex
Hamilton-Jacobi equations. Some remarkable results include the characterization
of the shock structure of the gradient of solutions \cite{Ev1}, construction
of invariant measures
in the spirit of weak KAM theory \cite{CGT1}, 
and homogenization in random media  \cite{ATY1, ATY2} (see also \cite{Gao}).
Better grasp of the vanishing discount problem for non-convex Hamiltonians
is essential to improving
our understanding of the nature of viscosity solutions of Hamilton-Jacobi
equations.

Before we proceed, we recall some elementary properties of  \eqref{e}  and  \eqref{de}. First, there exists a unique real constant, $\overline{H}$, such that \eqref{e} has viscosity 
solutions  \cite{LPV}. This constant is often called the {\em ergodic constant} or the {\em effective Hamiltonian}. However, in general,   \eqref{e} does not have a unique solution, not even up to additive constants. The lack of uniqueness is a central issue in the study of   
the asymptotic behavior of $u^\ep$ as $\ep \to 0$.
As \eqref{de} is strictly 
monotone with respect to $u^\ep$ for $\ep>0$, Perron's method gives the existence of  a unique viscosity solution,
$u^\ep$. By the coercivity of the Hamiltonian, we have that 
\begin{equation}\label{intro:apriori}
\|Du^{\ep}\|_{\Li(\T^n)}\le C \quad \text{for some} \ C>0 \ \text{independent of} \ \ep. 
\end{equation}
We fix  $x_0\in\T^n$. The preceding estimate implies that  
\[
\{u^{\ep}(\cdot)-u^{\ep}(x_0)\}_{\ep>0} 
\]
is uniformly bounded
and equi-Lipschitz continuous in \( \T^n\).
Therefore, by the Arzel\'a-Ascoli theorem, there exists a subsequence, 
$\{\ep_j\}_{j\in\N}$, with $\ep_j\to0$ as $j\to\infty$, and a function, $u\in C(\T^n)$, such that 
\begin{equation}\label{conv:sub}
\ep_{j}u^{\ep_j}\to-\overline{H}, 
\quad u^{\ep_j}-u^{\ep_j}(x_0)\to u\in C(\T^n),   
\end{equation}
uniformly in \( \T^n \) as \( j\to\infty\).
By a standard viscosity solution argument, we see that  $(u,\overline{H})$ solves \eqref{e}. However, the convergence in  \eqref{conv:sub} and the function
 $u$ may depend on the choice of the subsequence $\{\ep_j\}$. Thus,  the limit as $\ep \to 0$ of $u^\ep$ may not exist. 

Our primary goal is to study the \textit{selection
problem} for  \eqref{de}; 
that is, we wish to understand whether or not the limit as $\ep \to 0$ of $u^\ep$ exists  and,
if it does, what the characterization of this limit is. This problem was proposed in \cite{LPV} (see also \cite[Remark 1.2, page 400]{BCD}). It remained unsolved for almost
30 years. Recently, there was substantial progress  in the case of convex
Hamiltonians. First, a partial characterization of the possible limits was
given in  \cite{G2} in terms of the
Mather measures (see, for example, \cite{FaB, Man, M}). 
Then, the convergence to a unique limit and its characterization were established
 in \cite{DFIZ}  using weak KAM theory. 
Further selection problems including the case of degenerate viscous Hamilton-Jacobi
equations were addressed using the
 nonlinear adjoint method in \cite{MT6}.
Finally, an analogous convergence result for the case of  Neumann boundary
conditions was examined in  \cite{AAIY}.
The selection problem for possibly degenerate, fully nonlinear Hamilton-Jacobi-Bellman equations
was considered in \cite{IMT1, IMT2}. A related selection problem  was addressed
in  \cite{B, JKM} and   selection questions motivated
by  finite-difference schemes were examined in \cite{S}. In all these papers, the convexity
of the Hamiltonian was essential and no extensions to non-convex
Hamiltonians were offered. Thus, the
selection problem in the non-convex setting has yet to be studied.


Here, we develop methods to examine the selection problem for \eqref{de} for non-convex Hamilton-Jacobi equations. Our main technical device is a selection theorem for a class of nonlinearly discounted Hamilton-Jacobi equations, Theorem \ref{thm:main}. Although this theorem is of independent interest, we focus here on two main applications: the case of strictly quasiconvex Hamiltonians in Theorem \ref{thm:quasi}  and the case of double-well problems in Theorem  \ref{thm:double-well}. These results and the main assumptions are  stated in the next section. Next, in Section \ref{sec:general}, we introduce a generalized discounted approximation, examine its  convergence, and prove Theorem \ref{thm:main}.  
Our proof  is based on the method introduced in \cite{MT6}. 
Then, in Section \ref{sec:nonconvex}, we study strictly quasi-convex Hamiltonians
and prove Theorem
 \ref{thm:quasi}. 
 Next, in Section \ref{sec:dw}, we consider the double-well 
 Hamiltonian-Jacobi equation and prove Theorem \ref{thm:double-well}.
Finally,  in Section \ref{sec:ex}, we examine the convergence for a quasi-convex Hamiltonian with flat parts. 
The results in this section do not follow from the general theory developed
in Section \ref{sec:nonconvex} and they require a distinct approach.  In this final section, we discuss maximal subsolutions and the Aubry set.
In particular, we provide an answer to Question 12 in the list of open problems \cite{Bo} from the conference ``New connections between dynamical systems and PDEs" at the American Institute of Mathematics in 2003.

\section{Assumptions and main results}

Here, we
discuss the main assumptions used in the paper and present the main results.

Let
$G \in C^1(\T^n \times \R^n)$ and $f\in C^{2}(\T^n\times\R)$ satisfy
\begin{itemize}
\item[(A1)] 
{ 
uniformly for $x\in\T^n$, 
\[
\lim_{|p| \to \infty} \left(\frac{1}{2|p|} G(x,p)^2 - |D_xG(x,p)|\right) = +\infty;
\]}
\item[(A2)] 
$p \mapsto G(x,p)$ is convex; 
\item[(A3)] 
$f_{r}(x,r)>0$ for all $(x,r)\in\T^n\times\R$.
{
There exists $M>0$ such that, for all $x \in \T^n$,
\[
f(x,-M) \leq -G(x,0) \leq f(x,M).
\]
}
\end{itemize}
We consider the following generalization of  the discounted problem
\begin{equation}\label{eq:general}
f(x,\ep v^\ep)+G(x,Dv^\ep)=0 \quad\text{in} \,\, \T^n.
\tag{$\rm{GE}_\ep$} 
\end{equation}
Because of (A3), for  $\ep>0$, the left-hand side of \eqref{eq:general} is
strictly 
monotone in $v^\ep$. Therefore,
\eqref{eq:general} has a comparison principle. Furthermore, the
coercivity of $G$ given by (A1) implies that $\|Dv^\ep\|_{L^\infty}<C$ for
some constant, $C$, independent of $\ep$ (see Lemma \ref{lem:g1} below). 
Thus, arguing as before, we see that
there exists a constant, $c \in \R$, such that 
$\ep v^\ep\to-c$ in $C(\T^n)$ as $\ep\to0$. 
Accordingly, we consider the  following
 ergodic problem associated with \eqref{eq:general}:\begin{equation}
\label{ge}
 f(x,-c)+G(x,Dv)=0\quad\text{in} \,\, \T^n. 
\tag{GE}
\end{equation}
Without loss of generality, by replacing $f$ by $f_c(x,r)=f(x,-c+r),$ if
necessary, 
we can assume that $c=0$.

\begin{thm}\label{thm:main}
Assume {\rm(A1)--(A3)} hold. Let $v^\ep$ be the viscosity solution 
of \eqref{eq:general}.
Let $\cM$ be given by  \eqref{Mather-general}. Let 
 $\cE$ be the family of subsolutions $w$ of \eqref{ge} that satisfy\begin{equation}\label{class-sub}
\iint_{\T^n \times \R^n} f_r(x,0)w(x)
 \,d\mu \leq 0 \qquad \text{for all} \ \mu \in \cM.    
\end{equation}
Define
\[
v^0(x)=\sup_{w \in \cE} w(x). 
\]

Then, we have 
\begin{equation}\label{main:conv}
v^\ep(x) \to v^0(x), \quad\text{uniformly for} \ x\in\T^n \
\text{as} \ \ep\to0.   
\end{equation}

\end{thm}

A Hamiltonian, $H,   $  is {\it strictly quasi-convex}
if it satisfies the following assumption: 
\begin{itemize}
\item[(A4)]For any $a\in\R$ and $x\in\T^n$, the set 
$\{p \in \R^n\,:\, H(x,p)\le a\}$ is convex, and there exists a constant,
$\lam_0>0,$ such 
that 
\[
\lam_0^2 D_pH(x,p)\otimes D_pH(x,p)+\lam_0 D_{pp}^2H(x,p) \geq 0
\quad\text{for all} \ (x,p)\in\T^n\times\R^n.
\]
\end{itemize}
If the preceding assumption holds, we have that  $G(x,p):=e^{\lam_0H(x,p)}$
is a convex 
function of $p$. In addition to (A4), it is useful to introduce the  following growth
assumption on $G$.
\begin{itemize}
\item[(A5)]
{ $G(x,p)=e^{\lam_0H(x,p)}$ satisfies (A1).} 
\end{itemize}

%

\begin{thm}\label{thm:quasi}
Assume {\rm(A4)}\ and  {\rm\ (A5)} hold. Let  $u^\ep$ solve \eqref{de}. Then, $u^\ep$ solves \eqref{eq:general}
for  $f(x,r)=-e^{-\lam_0 r}$ for $(x,r) \in \T^n \times \R  $, and $G$\ as in Assumption 
  {\rm\ (A5).} Moreover, as $\ep\to 0$, $u^\ep$ 
converges uniformly  to the function 
$u^0$ determined by the conditions in Theorem \ref{thm:main}.  
\end{thm}

%
\begin{rem}\label{rem:A4}
While assumption (A4) is somewhat technical, it holds 
if, 
 for each fixed $x\in \T^n$
and each $s>\min H(x,\cdot)$, the level set $\{p\in \R^n\,:\, H(x,p)=s\}$
is a closed $(n-1)$-dimensional manifold whose second fundamental form is strictly
positive.

More precisely, the following assumption implies (A4):
\begin{itemize}
\item[(A4')] for each fixed $x\in \T^n$,
and each $s>s_0=\min H(x,\cdot)$, the level set $M_s=\{p\in \R^n\,:\, H(x,p)=s\}$
is a closed manifold of dimension  $n-1$  and, for each $p \in M_s$,
there exists $c=c(s)>0$ such that
\[
(B_p v)\cdot v \geq c|v|^2 \quad \text{for all} \ v \in T_p M_s,
\]
where $T_p M_s$ is the tangent plane to $M_s$ at $p$ and $B_p: T_p M_s \times
T_p M_s \to \R$ is the second fundamental form of $M_s$ at $p$.
Furthermore, there exists $\al>0$ such that, for each $p \in \partial
M_{s_0}=\partial \{p\in \R^n\,:\, H(x,p)=s_0\}$,
we have\[
D^2_{pp} H(x,p) \geq \al I_n,
\]
where $I_n$ is the identity matrix of size $n$
\end{itemize}
(see \cite[Section 9.7]{CGT1} for details).
The preceding condition is satisfied by a broad class of quasi-convex Hamiltonians of which a typical example  is  \[H(x,p)=K(|p|)+V(x),\] where $K:[0,\infty) \to \R$
is of class $C^2$ and satisfies
\[
K'(0)=0, \ K''(0)>0, \ \text{and} \ K'(s)>0 \ \text{for} \ s>0.
\]
\end{rem}

In Section \ref{sec:dw},
we consider an alternative approach to the non-convex, double-well Hamiltonian in one-dimensional space,
\begin{equation}\label{H-nonconvex}
H(x,p)=(|p+P|^2-1)^2 - V(x), 
\end{equation}
where $P\in\R$ and $V:\T\to\R$ is a continuous function
satisfying 
\begin{equation}\label{assume:V}
\min_{\T} V=0 \quad\text{and}\quad 
\max_{\T} V<1. 
\end{equation}
Although this Hamiltonian does not satisfy  (A4), we prove the following convergence result.

\ \ \ 
\begin{thm}\label{thm:double-well}
Let $H$\ be given by \eqref{H-nonconvex}. Let $u^\ep$ be the corresponding
solution of \eqref{de}  for a fixed $P \in \R$. 
Then,  there exists a  solution, $u^0 \in C(\T)$,  of 
\begin{equation}\label{eq:P-cell}
(|P+Du^0|^2-1)^2 - V(x) = \ol{H}(P) \quad \text{in} \ \T
\end{equation}
such that
\[
\lim_{\ep \to 0} \left (u^\ep + \frac{\ol{H}(P)}{\ep} \right)= u^0 \quad
\text{in} \ C(\T).
\]
\end{thm}
\begin{rem}
In the proof of the preceding theorem, we also obtain a characterization of limit $u_0$ that depends on the value of $P$ (see Section \ref{sec:dw}). 
\end{rem}


\section{A generalization of the discounted approximation
}\label{sec:general}

Here, we use the nonlinear adjoint method  \cite{Ev1} (see also \cite{T1}) and 
the strategy introduced in \cite{MT6} for the study of \eqref{eq:general} 
to investigate the limit $\ep\to 0$. 

\subsection{A regularized problem and the construction of Mather measures}\label{subsec:mather}

To study  \eqref{eq:general}, we introduce the following
regularized problem. For each $\eta>0$, we consider 
\begin{equation}
\label{ap}
f(x,\ep v^{\ep,\eta})+G(x,Dv^{\ep,\eta})=\eta^2\Del v^{\ep,\eta} \qquad \text{in} \ \T^n. 
\tag{${\rm A}_\ep^\eta$}
\end{equation}
\begin{lem}\label{lem:g1} 
Suppose that {\rm(A1) and (A3)}  hold. Then, 
there exists a constant, $C>0,$ independent of $\ep$ and $\eta$ such that,
for any solution 
 $v^{\ep,\eta}$ of \eqref{ap}, we have\begin{equation}
\label{lest}
\|Dv^{\ep,\eta}\|_{\Li(\T^n)}\le C.
\end{equation}
\end{lem}
{
\begin{proof}
Thanks to (A3), $-\ep^{-1}M$ and $\ep^{-1}M$ are a subsolution and a supersolution of \eqref{ap}, respectively.
We use the comparison principle to get $-\ep^{-1}M \leq v^{\ep,\eta} \leq \ep^{-1}M$ in $\T^n$.
In particular, $|f(x,\ep v^{\ep,\eta})| \leq C$ in $\T^n$ for $C = \max_{x \in \T^n, |r| \leq M} |f(x,r)|$.

Now, we prove the Lipschitz bound using Bernstein's method. 
First, we set $\phi:=|Dv^{\ep,\eta}|^2/2$. 
Differentiating \eqref{ap} in $x$ and multiplying by $Dv^{\ep,\eta}$, we get
\[
\ep f_r|Dv^{\ep,\eta}|^2+(D_xf+D_xG)\cdot Dv^{\ep,\eta}+D_pG\cdot D\phi
=\eta^2(\Del \phi-|D^2 v^{\ep,\eta}|^2). 
\]
Next, we choose $x_0\in\T^n$ such that $\phi(x_0)=\max_{\T^n}\phi$. 
According to (A3), we obtain 
\[
(D_xf+D_xG)\cdot Dv^{\ep,\eta}+\eta^2|D^2 v^{\ep,\eta}|^2\le 0 \quad\text{at} \ x_0\in\T^n.  
\] 
For $\eta<n^{-1/2}$,
we have\[\eta^2|D^2 v^{\ep,\eta}|^2\ge 
|\eta^2\Del v^{\ep,\eta}|^2=|f(x,\ep v^{\ep,\eta})+G(x, Dv^{\ep,\eta})|^2\ge 
\frac{1}{2}G(x_0,Dv^{\ep,\eta})^2-C
\]
for some $C>0$.  
Finally, to end the proof, we use (A1) to get \eqref{lest}. 
\end{proof}
}

The Lipschitz bound \eqref{lest} and the uniqueness 
of the solution of \eqref{eq:general} give that 
 $v^{\ep,\eta}\to v^{\ep}$ in $C(\T^n)$ as $\eta\to 0$.

Fix $x_0\in \T^n$ and let $\del_{x_0}$ denote the Dirac delta at $x_0$.
Next, we consider the linearization of \eqref{ap}
and introduce the corresponding adjoint equation
\begin{equation}
\label{aj}
\ep f_r(x,0)\theta^{\ep,\eta} - \Div (D_pG(x,Du^{\ep,\eta})\theta^{\ep,\eta})
=\eta^2 \Del \theta^{\ep,\eta} + \ep \del_{x_0}\qquad \text{in} \ \T^n.
\tag{${\rm AJ}_\ep^\eta$}
\end{equation}


Integrating \eqref{aj} in $\T^n$ and using the maximum principle, we get the following proposition.
\begin{prop}\label{prop:g2}
Let 
 $v^{\ep,\eta}$ solve \eqref{ap} and let $\theta^{\ep,\eta}$ solve \eqref{aj}. Then, we have 
\[\theta^{\ep,\eta}>0 \ \text{in} \ \T^n \setminus \{x_0\} \quad \text{and} \quad 
\int_{\T^n} f_r(x,0)\theta^{\ep,\eta}(x)\,dx=1 \quad \text{for any} \ \ep,\eta>0.\]
\end{prop}

In light of Lemma \ref{lem:g1} and of the Riesz Representation Theorem, there exists a 
non-negative Radon measure, $\nu^{\ep,\eta}$, on $\T^n \times \R^n$ such that 
\begin{equation}\label{def-nu-ep}
\int_{\T^n} \psi(x,Dv^{\ep,\eta}) \theta^{\ep,\eta}(x)\,dx
=\iint_{\T^n \times \R^n} \psi(x,p)\,d\nu^{\ep,\eta}(x,p), 
\quad \forall \, \psi \in C_c(\T^n \times \R^n). 
\end{equation}
Because $f_r(x,0)>0$ for all $x\in\T^n$ and because of Proposition \ref{prop:g2}, we have 
\[
\frac{1}{\max_{x\in\T^n}f_r(x,0)}
\le\iint_{\T^n\times\R^n}d\nu^{\ep,\eta}
\le \frac{1}{\min_{x\in\T^n}f_r(x,0)}. 
\]
Therefore, two subsequences, $\ep_j$ and $\eta_k$, exist with  $\ep_j \to 0$ and $\eta_k \to 0$  as $j, k\to\infty$. Probability measures, $\nu^{\ep_j}, \nu \in \cP(\T^n \times \R^n),$ also exist such that
\begin{equation}\label{mather}
\begin{array}{ll}
\nu^{\ep_j,\eta_k}\rightharpoonup \nu^{\ep_j}  
\quad&\text{as} \ \ k\to\infty,\\
\nu^{\ep_j}\rightharpoonup \nu  
\quad&\text{as} \ \ j\to\infty,
\end{array}
\end{equation}
weakly in the sense of measures.
The limit $\nu$ depends on $x_0$ and  on subsequences $\{\ep_j\}$ and $\{\eta_k\}.$ Thus, when we need to highlight this explicit dependence, we write it as $\nu=\nu(x_0,\{\ep_j\}, \{\eta_k\})$. 
Next,  we define the family of measures, $\cM\subset \cP,$ as 
\begin{equation}\label{Mather-general}
\cM=\bigcup_{x_0\in\T^n,\{\ep_j\}, \{\eta_k\}}\nu(x_0,\{\ep_j\}, \{\eta_k\}).
\end{equation}

\begin{prop}\label{prop:mather-property}
Suppose that {\rm (A1) and (A3)}\ hold. \ Then, for any $\nu\in\cM$, we have
\begin{itemize}
\item[(i)] \ 
$\displaystyle 
\iint_{\T^n\times\R^n}\big(D_pG(x,p)\cdot p-G(x,p)\big)\,d\nu(x,p)
=\iint_{\T^n\times\R^n}f(x,0)\,d\nu(x,p)$,
\item[(ii)] \ 
$\displaystyle 
\iint_{\T^n \times \R^n} D_pG(x,p)\cdot D\varphi \,d\nu(x,p)=0 
$ 
\quad for any $\varphi\in C^1(\T^n)$.  
\end{itemize}
\end{prop}
\begin{proof}
{
We first prove (i). 
Note that (A)$_{\ep}^{\eta}$ can be rewritten as 
\[
f(x,\ep v^{\ep,\eta})+D_pG(x,Dv^{\ep,\eta})\cdot Dv^{\ep,\eta}
-\eta^2\Del v^{\ep,\eta}
= D_pG(x,Dv^{\ep,\eta})\cdot Dv^{\ep,\eta}-G(x,Dv^{\ep,\eta}). 
\]
Let $\theta^{\ep,\eta}$ solve \eqref{aj}.
Multiplying the previous equation by $\theta^{\ep,\eta}$, integrating on $\T^n$, 
and using integration by parts, we get 
\begin{align*}
&\int_{\T^n}(D_pG(x,Dv^{\ep,\eta})\cdot Dv^{\ep,\eta}-G(x,Dv^{\ep,\eta}))\theta^{\ep,\eta}\,dx\\
=&\, 
\int_{\T^n}
f(x,\ep v^{\ep,\eta})\theta^{\ep,\eta}\,dx
-
\int_{\T^n}
\left(\Div(D_pG\theta^{\ep,\eta})+\eta^2\Del\theta^{\ep,\eta}\right)
v^{\ep,\eta}\,dx\\
=&\, 
\int_{\T^n}
(f(x,\ep v^{\ep,\eta})-\ep f_r(x,0)v^{\ep,\eta})\theta^{\ep,\eta}\,dx
-\ep v^{\ep,\eta}(x_0). 
\end{align*}
We use \eqref{def-nu-ep},  set $\eta=\eta_{j}$, and let $j\to \infty$. Finally, we   set    $\ep=\ep_{k} $ and let $k\to \infty$ 
to get (i). 

Next, we multiply \eqref{aj} by $\varphi\in C^1(\T^n)$. 
Then, we  integrate on $\T^n$, use integration by parts, and set
$\eta=\eta_j$ and  $\ep=\ep_k$. Finally, we take the limit $j\to \infty$ and then $k\to \infty$ to obtain (ii). 
}
\end{proof}

\subsection{Key estimates}\label{subsec:key-est}

Next, we use the nonlinear adjoint method to establish estimates for the solutions of  \eqref{eq:general}.
These estimates are essential ingredients of our convergence result for \eqref{eq:general}.

\begin{lem}\label{lem:key1}
Suppose that {\rm (A1)-(A3)} hold. Let $v^{\ep}$ solve \eqref{eq:general}. Then, as $\ep\to 0$, 
\[
\iint_{\T^n\times \R^n}f_r(x,0)v^{\ep}(x)\,d\nu(x,p)\le o(1) \quad
\text{for all} \  \nu \in \cM.  
\]
\end{lem}
\begin{proof}
Let $\gam \in C_c^\infty(\R^n)$ be a standard mollifier; that is,  $\gam \geq 0$,
$\supp \gam\subset\ol{B}(0,1)$ and 
$\|\gam\|_{L^{1}(\R^n)}=1$. 
For each $\eta>0$, set $\gam^\eta(y):=\eta^{-n} \gam(\eta^{-1}y)$ for $y\in \R^n$ and define
\[
\psi^\eta(x):=v^\ep\ast\gam^\eta(x)=\int_{\R^n} v^\ep(x-y)\gam^\eta(y) \,dy.
\]
Because  $G$ is convex and $\|Dv^\ep\|_{L^\infty}$ is bounded uniformly in $\ep$, we have 
\begin{equation}\label{convo}
G(x,D\psi^\eta) \leq G(x,Dv^\ep)+C\eta.
\end{equation}
Using Taylor's expansion, we get
\begin{equation}\label{Taylor}
f(x,\ep u^\ep)=f(x,0)+\ep f_r(x,0)u^\ep+o(\ep) \quad\text{as} \ \ep\to0.
\end{equation}
Using \eqref{convo}, \eqref{Taylor}, and the convexity of $G$, we obtain
\begin{align*}
0=&\,
f(x,\ep u^\ep)+G(x,Du^\ep)
\ge f(x,0)+\ep f_r(x,0) u^\ep+G(x,D\psi^\eta)-C\eta-o(\ep)\\
\ge&\,
f(x,0)+\ep f_r(x,0) u^\ep + G(x,p)+D_pG(x,p)\cdot(D\psi^\eta-p)-C\eta-o(\ep).  
\end{align*}
Next, we integrate the preceding inequality with respect to $d\nu(x,v)$ for $\nu\in\cM$ and use  properties (i) and (ii) of Proposition \ref{prop:mather-property} to conclude that $$
\iint_{\T^n \times \R^n}  f_r(x,0)u^\ep(x)\,d\nu(x,v) \leq \frac{C\eta}{\ep}+o(1).
$$
Finally, we let $\eta \to 0$ to achieve the desired result.
\end{proof}

\begin{lem}\label{lem:key2}
Suppose that {\rm(A1)-(A3)} hold. Let $w \in C(\T^n)$ be a subsolution of {\rm(GE)}. Then, we have 
\[
u^{\ep}(x_0) 
\geq w(x_0)-\iint_{\T^n\times\R^n} f_r(x,0)w(x) \,d\nu^{\ep}(x,p)
+o(1) \quad\text{as} \ \ep\to0
\] 
for all $x_0\in\T^n$, 
where $\nu^{\ep}$ is a weak limit in the sense of measures of a subsequence of $\nu^{\ep,\eta}$ as $\eta \to 0$.
\end{lem}
\begin{proof}
For $\eta>0$, let $w^{\eta}:=w\ast\gam^{\eta}$, where $\gam^{\eta}$ is a mollifier as 
in the proof of Lemma \ref{lem:key1}.  
Because $\eta^2|\Del w^\eta|\le C\eta$, we obtain 
\begin{equation}\label{w-eta}
f(x,0)+G(x,Dw^\eta) \leq \eta^2 \Del w^\eta +C \eta \qquad \text{in} \ \T^n. 
\end{equation}
Now, using  \eqref{Taylor}, we rewrite (A)$_\ep^\eta$ as 
\begin{equation}\label{eq:A}
f(x,0)+\ep f_r(x,0)u^{\ep,\eta}+o(\ep)+G(x,Du^{\ep,\eta}) 
=\eta^2\Del u^{\ep,\eta}. 
\end{equation}
Next, 
we subtract \eqref{w-eta} from \eqref{eq:A} to get
\begin{align*}
&C \eta+\ep f_r(x,0)u^{\ep,\eta}+o(\ep)
\ge
G(x,Dw^\eta)-G(x,Du^{\ep,\eta})-\eta^2 \Del(w^\eta-u^{\ep,\eta})\\
\ge&\, 
D_pG(x,Du^{\ep,\eta})\cdot D(w^\eta-u^{\ep,\eta})-\eta^2\Del(w^\eta-u^{\ep,\eta}). 
\end{align*}
Multiplying the preceding inequality by a solution, $\theta^{\ep,\eta}$, of  (AJ)$_{\ep}^{\eta}$, integrating on $\T^n$, 
and using integration by parts, we get 
\begin{align*}
&\int_{\T^n}\left(C \eta+\ep f_r(x,0)u^{\ep,\eta}+o(\ep)\right)\theta^{\ep,\eta}\,dx\\
\ge&\, 
-\int_{\T^n}\big(\Div (D_pG \theta^{\ep,\eta})+\eta^2\Del\theta^{\ep,\eta}\big)(w^\eta-u^{\ep,\eta})\,dx\\
=&\, 
\ep\int_{\T^n}\left(\del_{x_0}-f_r(x,0)\theta^{\ep,\eta}\right)(w^\eta-u^{\ep,\eta})\,dx. 
\end{align*}
Next, we rearrange the previous estimate and get
\[
u^{\ep,\eta}(x_0) 
\geq w^\eta(x_0)-\iint_{\T^n\times\R^n} f_r(x,0)w^\eta \,d\nu^{\ep,\eta}(x,p)
-\frac{C\eta+o(\ep)}{\ep}\iint_{\T^n\times\R^n} \,d\nu^{\ep,\eta}(x,p). 
\]
Finally, we set $\eta=\eta_{k}\to0. $ Thus,  $\nu^{\ep,\eta_k}\rightharpoonup \nu^{\ep}$ as measures. Taking the limit in the preceding inequality 
ends the proof.
\end{proof}

\subsection{Convergence} \label{subsec:conv1}
Here, we prove the selection theorem for  \eqref{eq:general}, Theorem \ref{thm:main}. This theorem substantially  extends the existing results for convex Hamiltonians and is the key technical device in the study of quasi-convex and double-well Hamiltonians.  

\begin{proof}[Proof of Theorem {\rm\ref{thm:main}}]
Let $\{\ep_j\}_{j\in\N}$ be any subsequence converging to $0$ such that $v^{\ep_j}$ 
 converges uniformly to a solution of \eqref{ge} as $j\to\infty$.
In view of Lemma \ref{lem:key1} and the definition of  $v^0$, we have that 
\begin{equation}\label{eq:m1}
v^0\ge \lim_{j\to\infty} v^{\ep_j}.
\end{equation}
Moreover, by Lemma \ref{lem:key2}, we get  
\[
\lim_{j\to\infty} v^{\ep_j}(x)\ge w(x)-\iint_{\T^n\times\R^n}f_r(x,0)w(x)\,d\nu(x,p)
\]
for any subsolution, $w$, of \eqref{e}. Next, we take $w=v^0$. Accordingly, using the definition of  $v^0$ again, 
we obtain 
\begin{equation}\label{eq:m2}
\lim_{j\to\infty} v^{\ep_j}\ge v^0.
\end{equation}
Thus, we combine \eqref{eq:m1} and \eqref{eq:m2} to get the desired result.
\end{proof}


\section{Strictly quasi-convex Hamiltonians}\label{sec:nonconvex}
Now, we use the results in the preceding section to investigate the selection problem for  strictly quasi-convex Hamilton-Jacobi equations and to prove Theorem
\ref{thm:quasi}.
For convenience, we assume that $\overline{H} =0$ in \eqref{e}. 

\begin{lem}\label{lem:exp}
A function, $u^\ep\in C(\T^n),$ solves \eqref{de} if and only if $u^\ep$ solves  \eqref{eq:general} 
for  $f(r)=-e^{-\lam_0 r}$,  $r \in \R  $, and $G$\ as in Assumption
  {\rm\ (A5).}
\end{lem}
\begin{proof}
Clearly, $u^\ep$ is a subsolution of \eqref{de} if and only if  for any $x\in \T^n$ and any $p \in D^+ u^\ep(x)$, we have
\begin{equation}\label{exp-1}
\ep u^\ep(x) + H(x,p)\leq 0.
\end{equation}
Moreover,  \eqref{exp-1} holds if any only if 
\begin{equation}\label{exp-2}
-e^{-\ep u^\ep(x)} + G(x,p)\leq 0.
\end{equation}
Arguing in a similar way for the supersolution case gives the result.      
\end{proof}

Theorem \ref{thm:quasi} is an immediate corollary of Theorem \ref{thm:main}. The proof of Theorem \ref{thm:quasi} follows. 

\begin{proof}[Proof of Theorem {\rm\ref{thm:quasi}}]
 By the preceding Lemma, $u^\ep$ solves \eqref{eq:general}.
It is clear that $f, G$ satisfy (A1)--(A3).  
Thus, we apply  Theorem \ref{thm:main} to obtain the last statement of the theorem. 
\end{proof}


\section{One-dimensional, non-convex, double-well Hamiltonians}
\label{sec:dw}
 
For each $P\in \R$, we consider the discounted Hamilton-Jacobi equation
\begin{equation}\label{eq:nonconvex}
\ep u^\ep+(|P+u^\ep_x|^2-1)^2 - V(x)=0 \quad\text{in} \ \T.
\end{equation}
As before, $\lim_{\ep \to 0} \ep u^\ep=-\ol{H}(P)\in\R$, where $\ol{H}(P)$ is the unique constant for which 
\[
(|P+u_x|^2-1)^2 - V(x)=\ol{H}(P) \quad\text{in} \ \T
\]
has a viscosity solution. 

Assumption \eqref{assume:V} means that  $V$ has a small oscillation; that is, $\osc(V):=\max_{\T}V-\min_{\T}V<1$.
Because the wells of $(|p|^2-1)^2$ have depth $1$, which is larger than $\osc(V)$, the effect of $V$\ on 
$\ol{H}(P)$ is localized. Moreover, from the results in \cite{ATY1}, the graph of $\ol{H}(P)$ follows and in shown in Figure \ref{Fig1}.
\begin{figure}
\includegraphics[width=70mm]{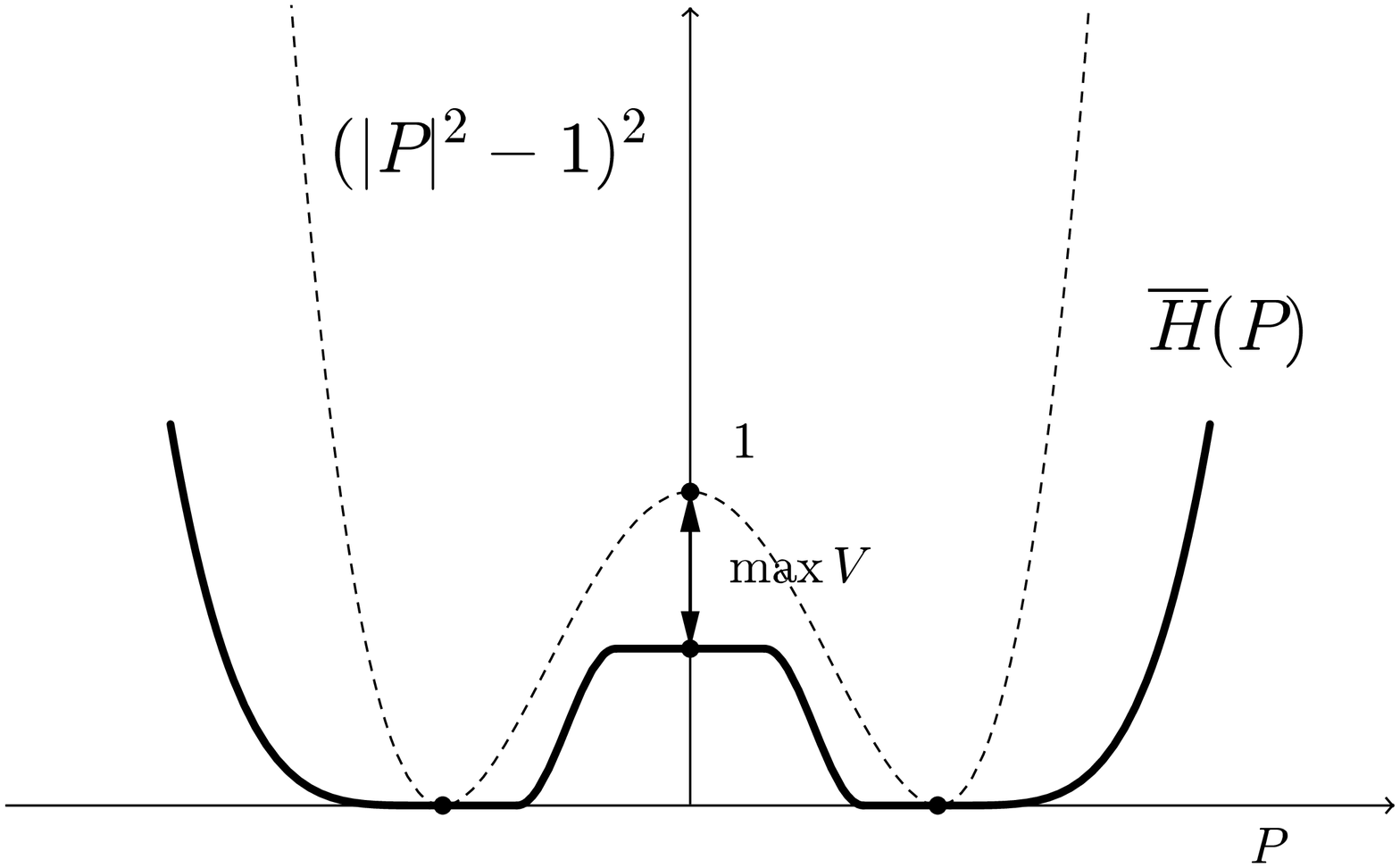}
\caption{The shape of $\ol{H}(P)$}\label{Fig1}
\end{figure}
As suggested by  Figure \ref{Fig1}, to prove Theorem \ref{thm:double-well}, we separately consider different cases 
according to the region where $P$\ lies.
We have the following a priori estimates that are essential in the proof of Theorem \ref{thm:double-well}.

\begin{prop}\label{prop:apriori}
Let $u^\ep$ solve \eqref{eq:nonconvex}. Consider the following three cases:
\begin{equation}
\label{cases}
{\rm(a)} \  |P|<1, \ \ol{H}(P)>0, \quad
{\rm(b)} \ |P|>1, \ \ol{H}(P)>0, \quad
{\rm(c)} \ \ol{H}(P)=0. 
\end{equation}
Then, in the viscosity sense,  we have,
\begin{itemize}
\item[{\rm (i)}] In case {\rm(a)},  for $\ep>0$ sufficiently small,
$|P+u^\ep_x|\le 1$ in $\T$. 
\item[{\rm (ii)}] In case {\rm(b)},  if $P>1$ then $P+u^\ep_x \ge 1$ in $\T$ for sufficiently small $\ep>0$ .
 If $P<-1,$ then $P+u^\ep_x \leq -1$ in $\T$ for sufficiently small $\ep>0$.
\item[{\rm (iii)}] In case {\rm(c)},  if $P>0$, then $P+u^\ep_x\ge 0$ in $\T$ for sufficiently small $\ep>0$.
If $P<0$, then $P+u^\ep_x\le 0$ in $\T$ for sufficiently small $\ep>0$. 
\end{itemize}
\end{prop}
\begin{proof}
First, we extend  $u^\ep$ periodically to $\R$.

To prove (i), we argue by contradiction.
Suppose that there exists $z\in\R$ such that $u^\ep$ is differentiable at $z$ and $|P+u^\ep_x(z)|>1$. 
Define  $f:\R\to\R$ by 
\[
f(x)=u^{\ep}(x)+P x+|x-z|. 
\]
Because $|P|<1$, there exists $x_0\in\R$ such that 
$f(x_0)=\min_{x\in\R}f(x)$ (see Figure \ref{barrier}, case (a)). Next, we prove that $x_0\not=z$. 
Indeed, by setting $q=P+u^\ep_x(z)$, we obtain 
\begin{align*}
f(z-\al q)
&=\, 
P(z-\al q)+u^\ep(z-\al q)+\al|q|\\
&=\, 
P(z-\al q)+u^\ep(z)-\al u^\ep_x(z) q+o(\al)+\al|q|\\
&=\, 
f(z)-\al\left(|q|(|q|-1)+\frac{o(\al)}{\al}\right), 
\end{align*}
which implies $f(z-\al q)<f(z)$ for a small $\al>0$ since $|q|>1$. Thus, $x_0\neq z$.

Because $ u^{\ep}$ is a viscosity supersolution  of  \eqref{eq:nonconvex} at $x=x_0$, we have
\[
V(x_0)\le\ep u^{\ep}(x_0)+\left(\left|P-\left(P+\frac{x_0-z}{|x_0-z|}\right)\right|^2-1\right)^2=\ep u^{\ep}(x_0). 
\] 
Since $\ep u^{\ep}(x_0)\to-\ol{H}(P)<0$, the above inequality yields a contradiction for small $\ep>0$. 
Thus, (i) holds.
\smallskip

Next, we prove (ii). We consider only the case when $P>1$ and argue by contradiction. If $P<-1$, the argument is analogous.  
Suppose that there exists $z\in\R$ such that 
$u^\ep$ is differentiable at $z$ and $P+u^\ep_x(z)<1$. 
Define the function $f:\R\to\R$ by 
\[
f(x):=u^{\ep}(x)+P x-|x-z|. 
\]
Because $P>1$, there exists $x_0\in [z,\infty)$ such that 
$f(x_0)=\min_{x\in [z,\infty)}f(x)$ (see Figure \ref{barrier}, case (b)). 
First,  we prove that $x_0>z$. We start by setting $q=P+u^\ep_x(z)$. Then, we have
\begin{align*}
f(z+\al)
&=\, 
P(z+\al)+u^\ep(z+\al)-\al\\
&=\, 
P(z+\al)+u^\ep(z)+\al u^\ep_x(z)+o(\al)-\al\\
&=\, 
f(z)+\al\left(q-1+\frac{o(\al)}{\al}\right). 
\end{align*}
Because $q<1$, the preceding identity implies that  $f(z+\al)<f(z)$ for a small $\al>0.$ Hence, $x_0>z$. Consequently, by the argument of the last part of the proof of (i), 
we get a contradiction. 
\begin{figure}[htbp]
  \begin{center}
    \begin{tabular}{c}
\hspace*{-36pt}
      \begin{minipage}{0.33\hsize}
        \begin{center}
          \includegraphics[clip, width=6cm]{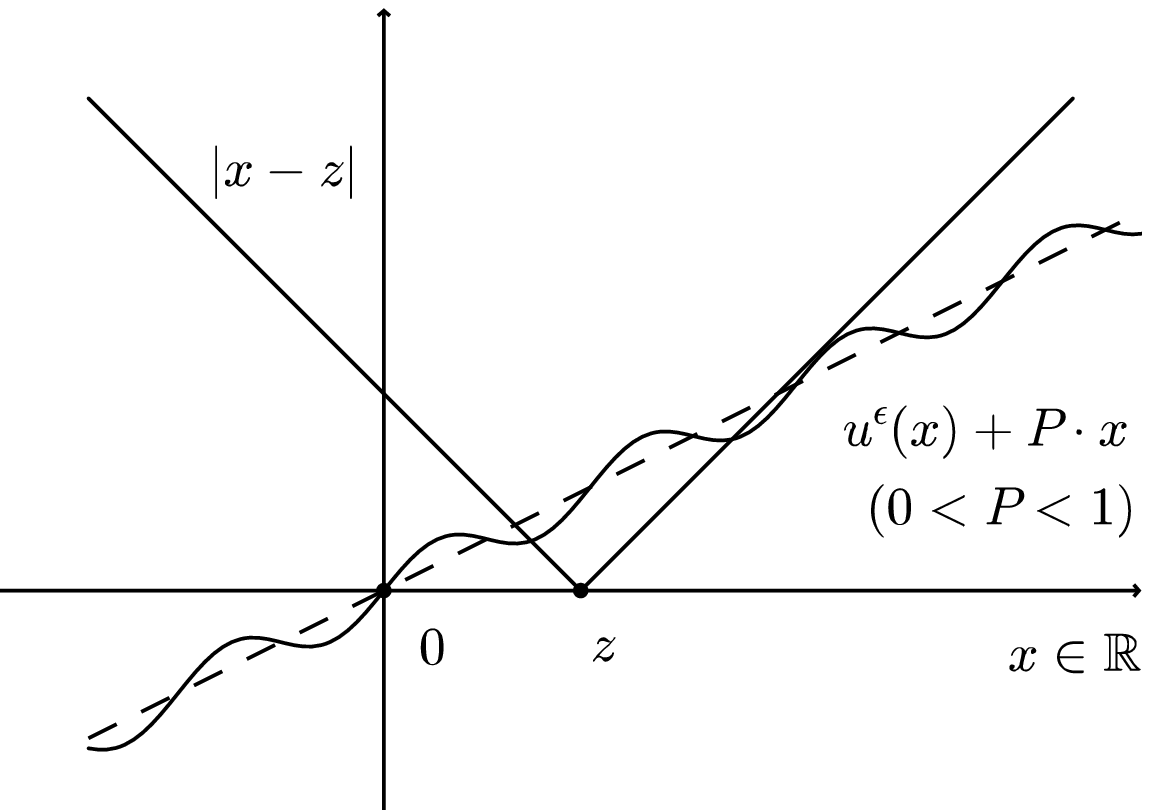}
        \hspace{1.6cm} {\bf Case (a)} 
        \end{center}
      \end{minipage}
\hspace*{48pt}
      \begin{minipage}{0.33\hsize}
        \begin{center}
          \includegraphics[clip, width=6cm]{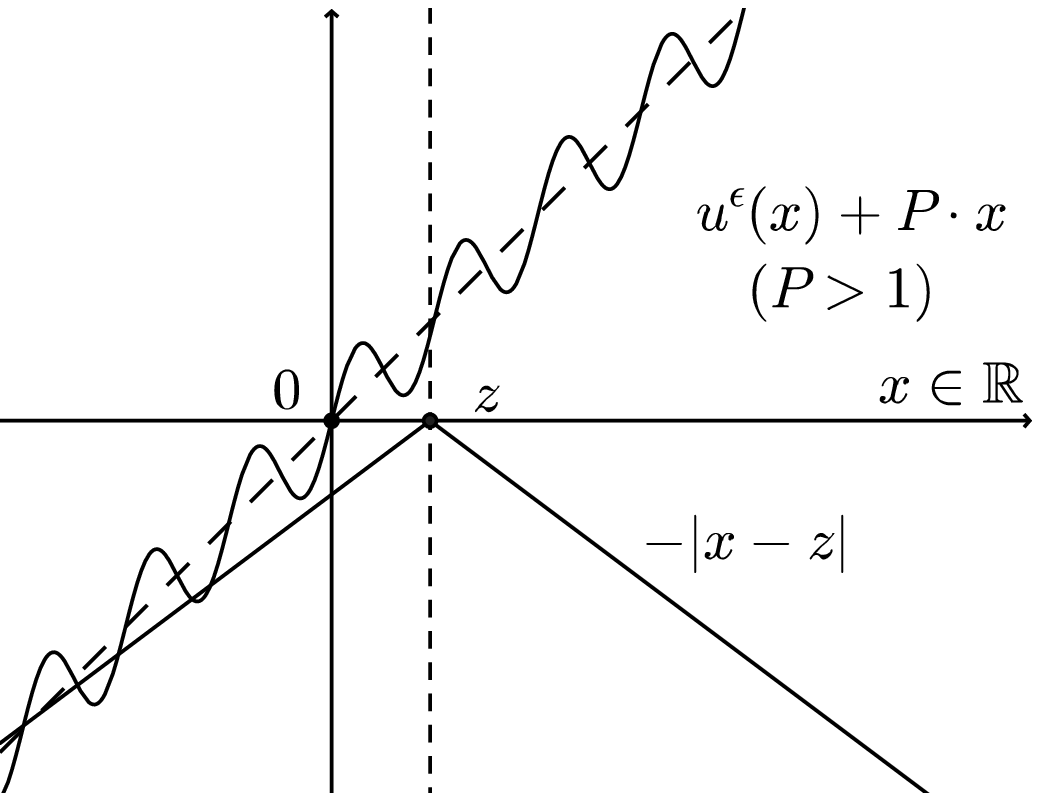}
         \hspace{1.6cm} {\bf Case (b)} 
        \end{center}
      \end{minipage}
    \end{tabular}
    \caption{}
    \label{barrier}
  \end{center}
\end{figure}

Finally, we prove (iii). We consider only the case when $P>0$ and prove  that $P+u_x^\ep\ge0$ in $\T$. The case when $P<0$ is analogous.

The proof proceeds by contradiction. First, if $P+u_x^{\ep}\le0$ for almost everywhere $x\in \T$, then 
\[
0=u^\ep(1)- u^\ep(0)= \int_0^1 u^\ep_x(x)\,dx \leq \int_0^1 (-P)\,dx= -P<0,
\]
which is a contradiction.
Therefore, we need to consider only the case when there exists 
$x_1, x_2\in\R$ with $x_1\not=x_2$ such that $u^\ep$ is differentiable at $x_1, x_2$ and 
\[
P+u_x^{\ep}(x_1)>0> P+u_x^{\ep}(x_2). 
\]
We can assume that $x_1<x_2$ without loss of generality. Otherwise, we replace $x_2$ by $x_2+k$ for some large enough $k \in \N$.

In view of \cite[Lemma 2.6]{ATY2}, there exists $x_3\in(x_1,x_2)$ such that 
$0\in P+D^{+}u^\ep(x_3)$. By the definition of the viscosity subsolution, 
we have 
\[
\ep u^\ep(x_3)+ |0^2-1|^2 - V(x_3) = \ep u^\ep(x_3)+1 - V(x_3) \leq 0, 
\]
which is a contradiction for sufficiently small $\ep>0$ s $\lim_{\ep \to 0} \ep u^\ep(x_3)=0$ and 
$\max V<1$. 
\end{proof}

\begin{rem}
By inspecting the proof of case (a)\ of the preceding proposition,  we see that the argument extends to arbitrary dimensions. 
In contrast, the proofs of the other two cases are one dimensional in nature;  we do not know how to generalize
them for higher dimensions. 
\end{rem}

Finally, we present the 
proof of Theorem \ref{thm:double-well}.

\begin{proof}[Proof of Theorem {\rm\ref{thm:double-well}}]
First, we  use Proposition \ref{prop:apriori} to transform \eqref{eq:nonconvex}  into \eqref{eq:general}. Then, we proceed as follows. 
We set $v^\ep := u^\ep +\ol{H}(P)/\ep$.  Thus, the ergodic constant becomes $0$.

In case (a), we use (i) of Proposition \ref{prop:apriori} to rewrite \eqref{eq:nonconvex} as 
\[
-\sqrt{V(x)+\ol{H}(P)-\ep v^\ep}-|P+v_x^\ep|^2+1=0\quad\text{in} \ \T. 
\]
Next, we set $G(x,p):=-|P+p|^2+1$.  Then, $G$ is concave (not convex) 
in $p$. This concavity, however, presents no additional difficulty because the proof of Theorem \ref{thm:main} can be adapted  easily to get the convergence. 

In case (b), we consider only the case when $P>1$ as the case when $P<-1$ is similar. In light of (ii) in Proposition \ref{prop:apriori},  we rewrite \eqref{eq:nonconvex} as 
\[
-\sqrt{V(x)+\ol{H}(P)-\ep v^\ep}+|P+v_x^\ep|^2-1=0\quad\text{in} \ \T. 
\]
Here, a direct application of Theorem \ref{thm:main} implies 
the convergence of $v^\ep$ as $\ep\to0$.

Finally, in case (c), we consider only the case when $P>0$. 
Because $P+u^\ep_x \ge 0$ in $\T$, only the positive branch ($p\geq 0$) of the graph of $(|p|^2-1)^2$ plays a role here.
Note that this branch is quasi-convex and satisfies (A4). 
Therefore, Theorem \ref{thm:quasi} gives the convergence of $u^\ep$ as $\ep\to0$. 
\end{proof} 

\subsection{A further generalization in one dimension}
The argument in the proof of Theorem \ref{thm:double-well} can be adapted to handle the case when the oscillation of the potential energy, $V$, is smaller than the depth
of any well of the kinetic energy, $H$.  
Thus, we  have the convergence of the discounted approximation.
We can generalize this idea as follows. 
Consider a Hamiltonian of the form
\[
H(x,p)= F(p) - V(x), 
\]
where $F(p)$ is the kinetic energy and $V(x)$ is the potential energy.
Assume that $-\infty=p_0<p_1<p_2 < \ldots < p_{2L+1}<p_{2L+2}=+\infty$ exists for some $L\in\N$ 
such that
\begin{itemize}
\item $\lim_{|p| \to \infty} F(p)= +\infty$, 

\item $F'(p_i)=0$ and $F''(p_i) \neq 0$ for $1 \leq i \leq 2L+1$, 

\item $F'(p)>0$ for $p \in (p_{2i+1}, p_{2i+2})$ for $0 \leq 1 \leq L$, 

\item $F'(p)<0$ for $p \in (p_{2i}, p_{2i+1})$ for $0 \leq 1 \leq L$.
\end{itemize}
Set 
\[
m= \min_{1 \leq i \leq 2L} |F(p_i) - F(p_{i+1})|.
\]
Assume that 
\[
\osc(V) < m.
\]
Under these assumptions, we can prove that the solution $u^\ep$ of (D)$_\ep$ 
converges to a solution of (E), which generalizes Theorem \ref{thm:double-well}. 

\begin{rem}\label{rem:large-osc}
If the oscillation of $V$ is larger than $m$, we cannot localize the convergence argument.  The qualitative behavior of $\ol{H}$ was examined in  \cite{ATY2}.
However, the characterization of the  convergence of the discounted approximation remains an open problem.
In this setting, for some values of $P$, we see that due to the non-convex nature of the gradient jumps, the problem cannot be 
transformed into an equation of the form \eqref{eq:general}.

\end{rem}

\section{An example: A quasi-convex Hamiltonian with flat parts}\label{sec:ex}

In this last section, we study a selection problem for which the results in Sections \ref{sec:general} and \ref{sec:nonconvex} do not apply.
We consider a continuous, piecewise $C^1$, quasi-convex Hamiltonian  that  has a level set with a flat part. Thus, (A4) does not hold and, therefore, we need an alternative approach. 

Assume that $n=1$.  For $p \geq 0$, let 
\[
F(p)=
\begin{cases}
p \quad &\text{for} \ 0 \leq p \leq 1,\\
1 \quad &\text{for} \ 1 \leq p \leq 2,\\
p-1 \quad &\text{for} \ p\geq 2.  
\end{cases}
\]
Consider the Hamiltonian
\begin{equation}\label{H-1d}
H(x,p)=F(|p|) - V(x),
\end{equation}
where $V$\ is as follows. First,  we select a sufficiently small $s>0$. Then, we 
 set
\begin{equation}\label{V-1d}
V(x)=
\begin{cases}
x \quad &\text{for} \ 0 \leq x \leq s,\\
2s-x \quad &\text{for} \ s \leq x \leq 2s,\\
0 \quad &\text{for} \ 2s\leq x \leq 1.
\end{cases}
\end{equation}

To study the effect of the flat part of $F$, we fix $P=3/2$.
For this value of $P$, we examine the  maximal subsolutions of \eqref{eq:P-cell}
and the discounted problem. 

\begin{figure}[htbp]
  \begin{center}
    \begin{tabular}{c}
\hspace*{-36pt}
      \begin{minipage}{0.33\hsize}
        \begin{center}
          \includegraphics[clip, width=6.5cm]{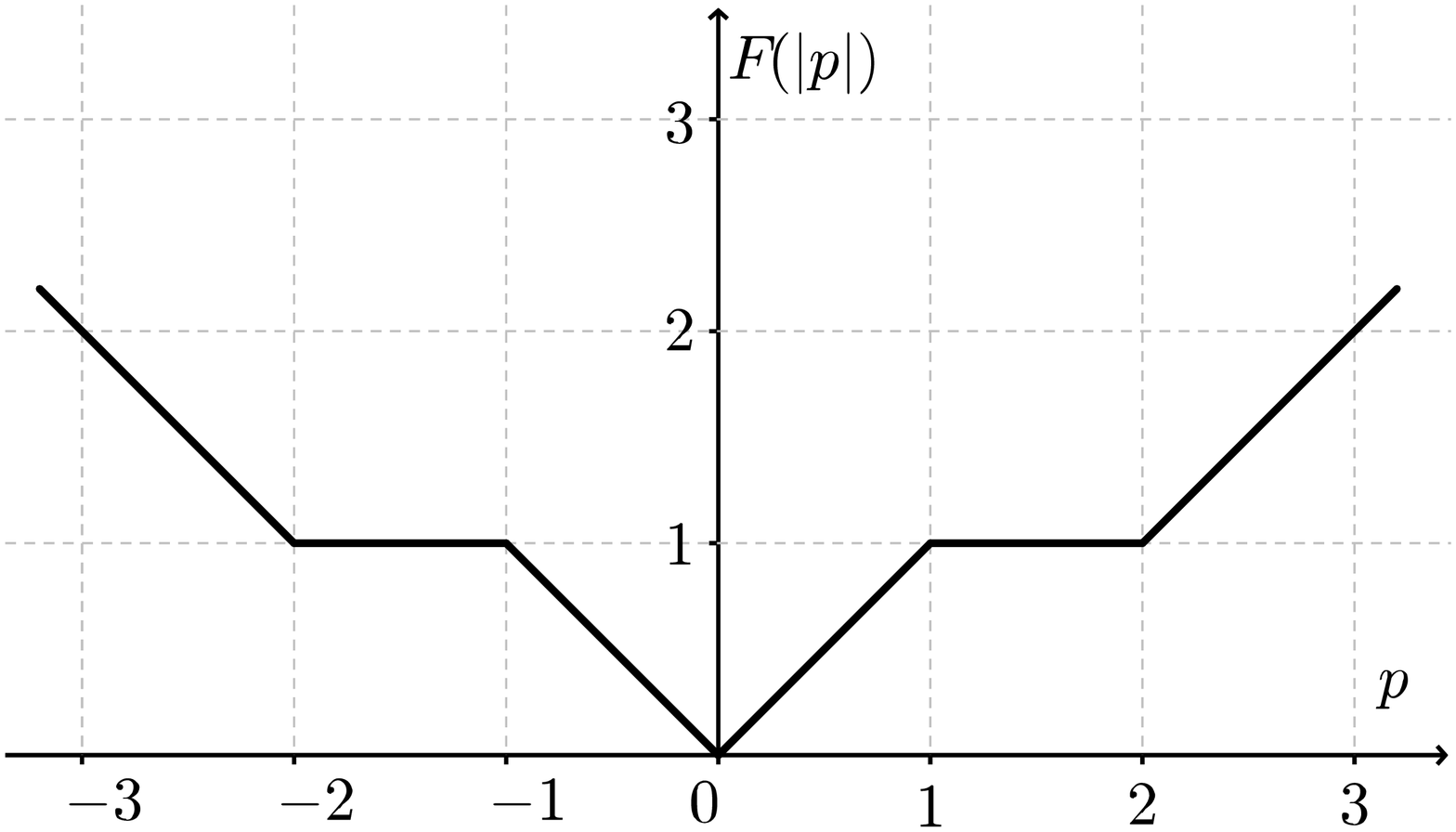}
        \hspace{1.6cm} Image of $F(|p|)$ 
        \end{center}
      \end{minipage}
\hspace*{48pt}
      \begin{minipage}{0.33\hsize}
        \begin{center}
          \includegraphics[clip, width=5cm]{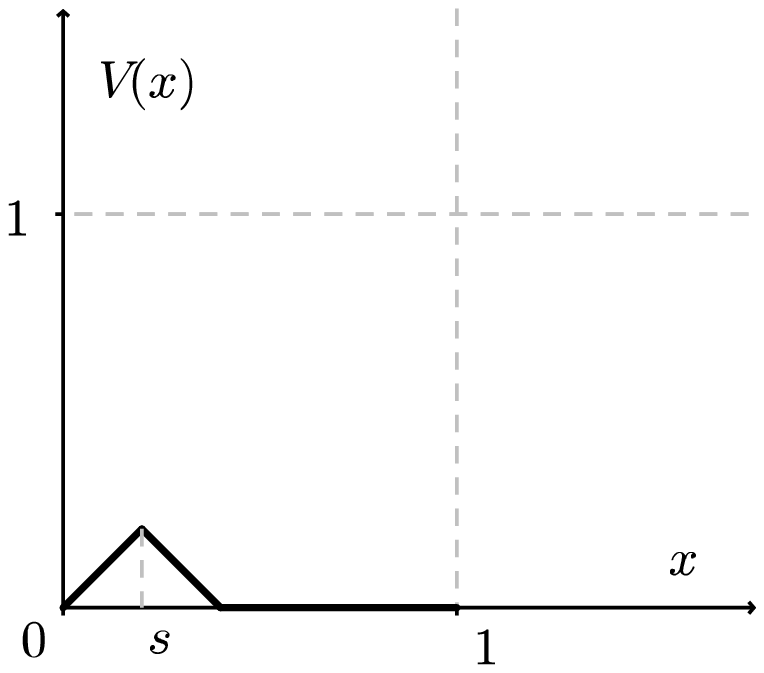}
         \hspace{1.6cm} Image of $V(x)$ 
        \end{center}
      \end{minipage}
    \end{tabular}
  \end{center}
\end{figure}

\subsection{Maximal subsolutions and the Aubry set at $P=3/2$}
First, we compute the effective Hamiltonian at $3/2$; that is, the unique value $\ol{H}(P)$ for which \eqref{eq:P-cell} has a viscosity solution. 
 
\begin{lem}\label{lem:flat-1}
Assume that \eqref{H-1d} and \eqref{V-1d} hold.
Then,   $\ol{H}(3/2)=1$.
\end{lem}

\begin{proof}
Choose a function $v \in C(\T)$ such that
\[
\begin{cases}
v_x= 1/2+ V \quad &\text{in} \ (0,2s),\\
v_x \in [-1/2,1/2] \quad &\text{in} \ (2s,1),\\
\int_0^1 v_x\,dx=0.
\end{cases}
\]
Clearly, $v$ is a viscosity solution to
\begin{equation}\label{cell:3-2}
F\left(\left|\frac{3}{2}+v_x\right|\right) - V(x)=1 \quad \text{in} \ \T.
\end{equation}
Thus,  $\ol{H}(3/2)=1$.
\end{proof}

Next, we define the corresponding maximal subsolutions.
First, we fix a vertex $y \in \T$ and set
\begin{equation}\label{max-sub}
S(x,y)=\sup\{ w(x)- w(y)\,:\, w \ \text{is a subsolution of \eqref{cell:3-2}}\}.
\end{equation}
Clearly, $S(y,y)=0$, $x\mapsto S(x,y)$ is a subsolution of \eqref{cell:3-2} in the whole torus, $\T$, and it is 
is a solution of \eqref{cell:3-2} in $\T \setminus \{y\}$.
Because $S(\cdot,y)$ is the largest subsolution $w$ of \eqref{cell:3-2} satisfying $w(y)=0$, we call it the maximal subsolution with vertex $y$.

 In the conference 
``New connections between dynamical systems and PDEs" at the American Institute
of Mathematics in 2003, Sergey Bolotin posed the following question (see \cite{Bo},    question 12 in the list of open problems):
\begin{quest}
Does there exist $y \in \T$ such that
$x \mapsto S(x,y)$ is a solution of \eqref{cell:3-2} in $\T$?
\end{quest}
{
The answer to the preceding question was found to be yes if $H$ is strictly quasiconvex   (see \cite{FS}). 
}
For the general nonconvex case,  this question has remained open.
Here, we answer no to this question (see also \cite[Example 12.7]{Fa-survey}).
More precisely, we offer the following proposition.

\begin{prop}\label{prop:A-empty}
For all $y \in \T$, $S(\cdot,y)$ is not a solution of \eqref{cell:3-2} in  $\T$.
\end{prop}

\begin{proof}
Fix $y \in \T$. Let $w:\T \to \R$ be a function such that $w(y)=0$ and 
\begin{equation}\label{A-w}
w_x(x)=
\begin{cases}
-\frac{7}{2} \qquad &\text{for} \ x \in \left(y-\frac{1}{8},y \right),\\
\frac{1}{2} \qquad &\text{for} \ x \in \left(y,y+\frac{7}{8}\right).
\end{cases}
\end{equation}
It is straightforward that $w$ is a subsolution of \eqref{cell:3-2} in the almost everywhere sense.  Hence, it is a  viscosity subsolution.
Therefore, $S(x,y) \geq w(y)$ for all $x \in \T$. In particular, this implies that
\begin{equation}\label{sub-S}
\left[-\frac{7}{2}, \frac{1}{2}\right] \subset D^- S(y,y).
\end{equation}
Next, we select $q=-\frac{3}{2} \in D^- S(y,y)$ and notice that
\[
F\left(\left|\frac{3}{2}+q\right|\right) - V(y) \leq F(0) = 0 <1.
\]
Consequently, $S(\cdot,y)$ is not a supersolution of \eqref{cell:3-2} at $y$.
\end{proof}

We observe that the maximal subsolution, $S(x,y)$, can be computed explicitly although we do not need this computation here. 
\begin{rem}
{
We recall that we can define the Aubry set for strictly quasiconvex Hamilton--Jacobi equations as 
the set of all points, $y$, such that $S(\cdot,y)$ is a solution on $\T^n$
(see \cite{FS} for the details).
Proposition \ref{prop:A-empty} implies that if we define the Aubry set 
in the same way, it is empty. 
However, this does not contradict the results in \cite{FS} as the Hamiltonian of the example in this section violates an assumption of the strictly quasiconvexity. 
This fact indeed highlights} a significant difference between convex and non-convex cases.  Therefore, if an analog of the Aubry set exists, it has to be defined in a different way. In the  general non-convex case, we can construct Mather measures  \cite{CGT1}
using the nonlinear adjoint method.
When the Hamiltonian is strictly
quasiconvex, 
these measures
are invariant under the Hamiltonian flow. 
Moreover, the Mather measures are supported in a subset of the Aubry set called the Mather set. This, of course, cannot hold if the Aubry set is empty.  Besides, in the general non-convex case, Mather measures may not be invariant under the Hamiltonian flow, and   the loss of invariance is encoded in dissipation measures that record the gradient jump structure \cite{CGT1}.
\end{rem}

\subsection{Discounted approximation at $P=3/2$}
Finally, we consider the discounted approximation problem for $P=3/2$. 
\begin{equation}\label{eq:flat}
\ep u^\ep + F\left(\left|\frac{3}{2}+u^\ep_x\right|\right) - V(x)=0 \quad \text{in} \ \T.
\end{equation}

\begin{prop}\label{thm:flat-flat}
There exists a solution of \eqref{cell:3-2},  $u^0 \in C(\T)$, such that
\[
\lim_{\ep \to 0} \left( u^\ep + \frac{1}{\ep} \right) = u^0 \quad \text{in} \ C(\T).
\]
\end{prop}

\begin{proof}
Let $v^\ep = u^\ep +1/\ep$. Then, $v^\ep$ solves
\begin{equation}\label{eq:v-ep}
\ep v^\ep + F\left(\left|\frac{3}{2}+v^\ep_x\right|\right) =1 + V(x) \quad \text{in} \ \T.
\end{equation}
Next, we give an explicit construction for  $v^\ep$.

\noindent {\bf Step 1.} Set
\[
v^\ep(x) = e^{-\ep x} \int_0^x e^{\ep r} \left(\frac{1}{2} + V(r) \right)\,dr \quad \text{for} \ x \in (0,a^\ep),
\]
where $a^\ep$ is a number to be chosen such that $a^\ep \in (s,2s)$ and $v^\ep_x({a^\ep}-)=1/2$.

It is clear that
\begin{itemize}
\item[-] $v^\ep(0)=0$ and $v^\ep_x(0+)=1/2$.

\item[-] $v^\ep(x) \leq x$ and
\[
v^\ep_x(x) = \frac{1}{2}+V(x) - \ep v^\ep(x).
\]

\item[-] In particular, for $0<x<s$, we have $v^\ep_x(x) \geq 1/2$ and thus
\begin{equation}\label{eq:step1}
\ep v^\ep(x)+ F\left(\left|\frac{3}{2}+v^\ep_x\right|\right) = \ep v^\ep(x) + \left(\frac{1}{2}+v^\ep_x\right)=1+V(x).
\end{equation}

\item[-] $v^\ep$ is increasing and always $\ep v^\ep=O(\ep)$. We choose $a^\ep \in (s,2s)$ such that $\ep v^\ep(a^\ep)= V(a^\ep)$.
Then, $\lim_{\ep \to 0} a^\ep = 2s$ and $v^\ep_x(a^\ep-)=1/2$. Clearly, \eqref{eq:step1} holds for all $x\in (0,a^\ep)$.
\end{itemize}

\noindent {\bf Step 2.} Define
\[
v^\ep(x) = e^{\ep(a^\ep-x)} v^\ep(a^\ep) + e^{-\ep x} \int_{a^\ep}^x e^{\ep r} \left(-\frac{1}{2} + V(r) \right)\,dr \quad \text{for} \ x\in (a^\ep, b^\ep),
\]
where $b^\ep>2s$ is a number to be chosen later. We have that
\begin{itemize}
\item[-] $v^\ep_x({a^\ep}+)=-1/2$.

\item[-] $v^\ep$ is decreasing in $(a^\ep, b^\ep),$ and
\[
v^\ep_x(x) = -\frac{1}{2} + V(x) - \ep v^\ep(x).
\]
\item[-] We argue that, for $x\in (a^\ep,2s)$, we have $\ep v^\ep(x) \geq V(x)$. This is correct as $\ep v^\ep(a^\ep) = V(a^\ep)$ and $\ep v^\ep_x(x) \geq -1 = V'(x)$ in $(a^\ep,2s)$.
Thus, $v^\ep_x(x) \leq -1/2$ in $(a^\ep,2s)$ and also $v^\ep(2s)>0$.

\item[-] Pick $b^\ep>2s$ to be the smallest number such that $v^\ep(b^\ep)=0$. Then, $v^\ep_x(b^\ep-)=-1/2$, and, for $x\in (a^\ep, b^\ep)$, we always have $v^\ep_x(x) \leq -1/2$ and
\[
\ep v^\ep(x)+ F\left(\left|\frac{3}{2}+v^\ep_x\right|\right) = \ep v^\ep(x) + \left(\frac{3}{2}+v^\ep_x\right)=1+V(x).
\]
\end{itemize}

\noindent {\bf Step 3.} For $x \in (b^\ep,1)$, we set $v^\ep(x)=0$. As $V=0$ in $(b^\ep,1)$, we have that, for $x \in (b^\ep,1)$,
\[
\ep v^\ep(x) +F\left(\left|\frac{3}{2}+v^\ep_x\right|\right) = F\left(\frac{3}{2}\right) = 1 = 1 + V(x).
\]
From the preceding three steps of the construction, $v^\ep$ is $1$-periodic.
To check that $v^\ep$ solves \eqref{eq:v-ep}, we only need to check the definition of viscosity solutions 
at the points where there are gradient jumps.  These points are $x=0, a^\ep, b^\ep$.
As we have $F(p)=1$ for $p\in [1,2]$, the verification at these three points is obvious.

Now, we are concerned with the convergence of $v^\ep$ as $\ep \to 0$.
As discussed in Step 1, we have that $\lim_{\ep \to 0} a^\ep =2s$. 
Set $b=\lim_{\ep \to 0} b^\ep$.
We use  the explicit formula of $v^\ep$ to get that $v^\ep \to u^0$, uniformly in $\T$, where $u^0$ satisfies
\[
\begin{cases}
u^0(0)=0\\
(u^0)'(x)=\frac{1}{2} + V(x) \quad &\text{in} \ (0,2s),\\
(u^0)'(x)=-\frac{1}{2} \quad &\text{in} \ (2s,b),\\
u^0 \equiv 0 \quad &\text{on} \ [b,1].
\end{cases}
\]
Finally, we see that $u^0$ solves \eqref{cell:3-2}.
\end{proof}

\end{document}